\documentclass[12pt,reqno,a4paper]{amsart}
\usepackage{blindtext}
\usepackage{fullpage}
\usepackage{mathtools}
\usepackage{longtable}
\usepackage{amsmath,amssymb,amsthm}
\usepackage{amscd}
\usepackage{bm}
\usepackage{hyperref}   
\usepackage{dsfont}
\usepackage{enumerate}
\usepackage{epsfig}
\usepackage{float, graphicx}
\usepackage{latexsym, amsxtra}
\usepackage{mathrsfs}
\usepackage{multicol}
\usepackage[normalem]{ulem}
\usepackage{psfrag}
\usepackage[parfill]{parskip}
\usepackage{stmaryrd}
\usepackage{tikz}
\usepackage[T1]{fontenc}
\usepackage{url}
\usepackage{verbatim}
\usepackage{indentfirst}
\usepackage{tikz-cd}
\usepackage{booktabs}
\usepackage{svg}

\usepackage{mathtools}

\usepackage{caption} 

\makeatletter
\def\thm@space@setup{%
	\thm@preskip=2ex \thm@postskip=2ex
}
\makeatother

\hypersetup{hidelinks}

\newtheorem{thm}{Theorem~}[section]
\newtheorem{lem}[thm]{Lemma~}

\newtheorem{prop}[thm]{Proposition~}

\newtheorem{conj}[thm]{Conjecture~}

\theoremstyle{remark}
\newtheorem{rmk}[thm]{Remark~}

\theoremstyle{definition}  
\newtheorem{defn}[thm]{Definition~}

\newcommand{\CC}{\mathbb{C}}
\newcommand{\ZZ}{\mathbb{Z}}

\newcommand{\LL}{\mathbb{L}}
\newcommand{\PP}{\mathbb{P}}

\newcommand{\HH}{\mathbb{H}}
\newcommand{\NN}{\mathbb{N}}

\newcommand{\V}{\mathcal{V}}

\title{The {N/D}-Conjecture for Nonresonant Hyperplane Arrangements}

\author[B. Xie]{Baiting Xie}
\address{Tsinghua University, China}
\email{xbt23@mails.tsinghua.edu.cn}

\author[C. Yu]{Chenglong Yu}
\address{Center for Mathematics and Interdisciplinary Sciences, Fudan University and
Shanghai Institute for Mathematics and Interdisciplinary Sciences (SIMIS), Shanghai, China}
\email{yuchenglong@simis.cn}
\date{}

\begin{document}

\begin{abstract}
	This paper studies Bernstein--Sato polynomials $b_{f}$ for homogeneous polynomials $f$ of degree $d$ with $n$ variables. It is open to know when $-{n\over d}$ is a root of $b_{f}$. For essential indecomposable hyperplane arrangements, this is a conjecture by Budur, Musta\c{t}\u{a} and Teitler and implies the topological monodromy conjecture for arrangements. By a criterion of Walther, it suffices to show that a certain differential form defines a nonzero class in the top cohomology of the Milnor fiber. We use this criterion to verify the $n\over d$-conjecture for weighted hyperplane arrangements satisfying the nonresonant condition.
\end{abstract}

\maketitle

	\section{Introduction}\label{sec: intro}

We first recall the definition of the Bernstein--Sato polynomials. Let $ R = \CC[x_{1},\cdots,x_{n}] $ be the polynomial ring over $\CC$, and $ D = R\langle \partial_{1},\cdots ,\partial_{n} \rangle $ be the Weyl algebra. Let $ s $ be a formal variable that commutes with all the $ x_{i},\partial_{i} $.  Let $ f \in \CC[x_{1},\cdots,x_{n}] $ be a non-constant polynomial. Recall that $ D[s] $ can act on $ R_{f}[s]  f^{s} $ as follows:
\begin{equation*}
	\partial_{i}\left(\frac{g}{f^{j}}\cdot f^{s}\right) = \partial_{i}\left(\frac{g}{f^{j}}\right)\cdot f^{s} + s\partial_{i}f\cdot\frac{g}{f^{j+1}}\cdot f^{s}.
\end{equation*}

\begin{defn}\label{def: b-function}
	The (global) Bernstein--Sato polynomial of $ f $ is defined to be the monic polynomial $ b_{f}(s) \in \mathbb{C}[s] $ of the smallest degree such that there exists some $ P(s) \in D[s] $ satisfying that
	\begin{equation*}
		P(s)  f^{s+1} = b_{f}(s)  f^{s}.
	\end{equation*}
\end{defn}

The existence of the Bernstein--Sato polynomial is a nontrivial result due to Bernstein \cite{bernshtein1972analytic}. There is also a local analytic version: for any point $x \in \CC^n$, replacing $R$ and $D$ by the ring $\mathcal O_{\CC^n,x}$ of germs of holomorphic function and its ring of differential operators defines the local Bernstein–Sato polynomial $b_{f,x}(s)$, whose existence was proved by Bj\"{o}rk \cite{Bjo74}. The global and local Bernstein--Sato polynomials satisfy
\[
b_f(s)=\operatorname{lcm}_{x\in f^{-1}(0)}b_{f,x}(s).
\]
In particular, if $f$ is homogeneous, then $b_f(s)=b_{f,0}(s)$.

The roots of the Bernstein--Sato polynomials are of particular interest in algebraic geometry. Malgrange \cite{malgrange1975polynome} proved that the roots of the local
Bernstein--Sato polynomial are negative rational numbers in the
isolated-singularity case, and Kashiwara established the same
conclusion for arbitrary germs of holomorphic functions in \cite{kashiwara1976b}. After that, Malgrange \cite{malgrange1983polynomes} and Kashiwara \cite{kashiwara1983vanishing} independently proved that the exponentials of the
roots of $b_{f,x}(s)$ recover the monodromy eigenvalues occurring
on the cohomology of local Milnor fibers at points of $f^{-1}(0)$
arbitrarily close to $x$.

Another direction is to relate these roots to the poles of
certain zeta functions. Following Igusa's monodromy conjecture
for $p$-adic local zeta functions \cite{Igu88}, Denef and Loeser
formulated its topological counterparts in
\cite[\S~3.3]{denef1992caracteristiques}. More specifically, there are three versions of the topological monodromy conjecture. For any point $x \in f^{-1}(0)$, let $Z_{\mathrm{top},f,x}(s)$ be the local topological zeta function of $f$ at $x$.  The weak version predicts that for every pole $c$ of $Z_{\mathrm{top},f,x}(s)$, the number
$\exp(2\pi i c)$ is an eigenvalue of the monodromy action on
$H^q(F_{f,y},\mathbb C)$ for some $q$ and points
$y\in f^{-1}(0)$ arbitrarily close to $x$. The usual version predicts that $b_{f,x}(c)=0$. And the strong version predicts that $b_{f,x}(s)Z_{\mathrm{top},f,x}(s)$ has no poles. In this paper, the topological monodromy conjecture refers to the statement $b_{f,x}(c)=0$. Early cases of this conjecture follow from
Loeser's work on $p$-adic local zeta functions for plane
curves \cite{Loe88} and certain Newton nondegenerate
polynomials under additional hypotheses \cite{Loe90};
see \cite[Remarks~3.3.3]{denef1992caracteristiques}
for the topological consequences. However, the topological monodromy conjecture remains open in general.

In this paper, we focus on the case where $ f $ is a central hyperplane arrangement. In this setting, Budur, Musta\c{t}\u{a} and Teitler \cite{budur2011monodromy} proved
the weak topological monodromy conjecture and reduced the topological monodromy conjecture to the so-called $ \frac{n}{d} $-conjecture. Before stating the latter, we first recall some basic notation for hyperplane arrangements.

Recall that a (reduced) central hyperplane arrangement $ A $ in $ \CC^{n} $ is a finite collection of pairwisely non-collinear functionals $ \{L_{1},\cdots,L_{r}\} \subset (\CC^{n})^{*} $. We also call homogeneous polynomials $f$ of the form $ \prod\limits_{j=1}^{r} L_{j}^{a_{j}} (a_{j} \in \ZZ_{>0}) $ central hyperplane arrangements. Note that $b_f(s)=b_{f,0}(s)$ since $f$ is homogeneous. So in the sequel we use the global Bernstein--Sato polynomial $b_f(s)$.

We call a subset $ B $ of a linear space $ V $ indecomposable if $ B \neq \emptyset $ and for any nontrivial decomposition $ V = W_{1} \oplus W_{2} $ we always have $ B \nsubseteq W_{1} \cup W_{2} $. A central hyperplane arrangement $ A $ (or $ f $) is called indecomposable if it is indecomposable as a subset of $ (\CC^{n})^{*}$. In other words, the polynomial $ f $ does not admit a factorization
\begin{equation*}
	f(x_{1},\cdots,x_{n}) = g_{1}(x_{1},\cdots,x_{i})g_{2}(x_{i+1},\cdots,x_{n}),\ 1 \leq i \leq n-1,
\end{equation*}
under any linear transformation of coordinates. Note that such types of central hyperplane arrangements are called essential and indecomposable in \cite{budur2011monodromy}.

The $ \frac{n}{d} $-conjecture by Budur, Musta\c{t}\u{a} and Teitler is stated as follows: 
\begin{conj}\label{conj: n over d}
	Let $ f $ be an indecomposable central hyperplane arrangement of degree $ d $ in $ \CC^{n} $. Then $ -\frac{n}{d} $ is a root of the Bernstein--Sato polynomial $ b_{f}(s) $.
\end{conj}

By \cite{budur2011monodromy}, this conjecture is sufficient for the topological monodromy conjecture of central hyperplane arrangements.

Conjecture \ref{conj: n over d} has been verified in several important cases. Budur, Saito and Yuzvinsky proved it for reduced arrangements with \(n\leq3\) \cite[Theorem 2(ii)]{budur2011local}. Saito subsequently established the stronger, multiplicity-sensitive topological monodromy statement for reduced arrangements in three variables \cite[Theorem 3]{saito2016bernstein}. Walther proved Conjecture 1.2 for tame arrangements, hence settling it for \(n\leq3\) \cite[Theorem 1.4]{walther2017jacobian}.

The problem becomes more challenging in higher dimensions. One possible approach is to characterize, for a fixed underlying reduced arrangement $ \{L_{1},\cdots,L_{r}\} $, the exponents $ a_{j} \in \ZZ_{>0} $ for which the Conjecture \ref{conj: n over d} holds for $ f = \prod\limits_{j=1}^{r} L_{j}^{a_{j}} $. A classical result from \cite{budur2011local} is the following. 
\begin{thm}[{\cite[Theorem 2(1)]{budur2011local}}]\label{thm: log threshold}
	Let $ f = \prod\limits_{j=1}^{r} L_{j}^{a_{j}} (a_{j} \in \ZZ_{>0} ) $ be an indecomposable central hyperplane arrangement of degree $ d $ in $ \CC^{n} $. Denote $ \{ L_{1},\cdots,L_{r} \} $ by $ A $ and define $ \mathcal{L} $ to be the set of nonzero proper linear subspaces $ W $ in $ (\CC^{n})^{*} $ such that $ A \cap W $ is indecomposable in $ W $. Suppose that
	\begin{equation*}
		\text{For all } W\in\mathcal{L},\  -\frac{n}{d}(\sum\limits_{L_{j} \in W}a_{j}) + \dim W \geq 0.
	\end{equation*}
	Then Conjecture \ref{conj: n over d} holds for $ f $.
\end{thm}

Recently, Shi and Zuo proved that Conjecture \ref{conj: n over d} holds for hyperplane arrangements with generic multiplicities.
\begin{thm}[{\cite[Theorem 1.7]{shi2024variation}}]\label{thm: generic multiplicity}
	Let $ A = \{L_{1},\cdots,L_{r}\} $ be an indecomposable reduced central hyperplane arrangement in $ \CC^{n} $. Then there exists a proper analytic subset $ S $ of $ \{(\alpha_{1},\cdots,\alpha_{r}) \in \CC^{r} \mid \alpha_{1}+\cdots+\alpha_{r} = 1\} $ such that for any $ (a_{1},\cdots,a_{r}) \in \ZZ_{>0}^{r} $, if $ (\frac{a_{1}}{a_{1}+\cdots+a_{r}},\cdots,\frac{a_{r}}{a_{1}+\cdots+a_{r}}) \notin S $, then Conjecture \ref{conj: n over d} holds for $ f $.	
\end{thm}

In this paper, we prove the Conjecture \ref{conj: n over d} under certain nonresonant conditions. The framework of our proof follows the approach of computing cohomology via the algebraic de Rham theorem, as developed in \cite{budur2011local} and \cite{saito2016bernstein}. 

\begin{thm}\label{thm: main}
	Let $ f = \prod\limits_{j=1}^{r} L_{j}^{a_{j}} (a_{j} \in \ZZ_{>0} ) $ be an indecomposable central hyperplane arrangement of degree $ d $ in $ \CC^{n} $. Denote $ \{ L_{1},\cdots,L_{r} \} $ by $ A $ and define $ \mathcal{L} $ to be the set of nonzero proper linear subspaces $ W $ in $ (\CC^{n})^{*} $ such that $ A \cap W $ is indecomposable in $ W $. Suppose that
	\begin{equation}\label{eqn: nonresonant}
		\text{For all } W\in\mathcal{L},\ -\frac{n}{d}(\sum\limits_{L_{j} \in W}a_{j}) + \dim W \notin \ZZ_{>0}. \tag{R}
	\end{equation}
	Then Conjecture \ref{conj: n over d} holds for $ f $.
\end{thm}

\begin{rmk}
	The set $ \mathcal{L} $ corresponds to the set of all dense edges of $ A $. The condition \eqref{eqn: nonresonant} is used in \cite{schechtman1995local} to prove vanishing results about the cohomology of rank-one local systems on complements of hyperplane arrangements. A similar condition also appears in \cite{esnault1992cohomology}. Assuming the nonvanishing of certain cohomology classes, Budur, Saito and Yuzvinsky proved Conjecture \ref{conj: n over d} under a similar nonresonant condition
	, see \cite[Theorem 2.5]{budur2011local}. Our result removes the assumption on nonzero cohomology classes.
\end{rmk}

Note that Theorem \ref{thm: main} gives an explicit combinatorial condition in terms of the intersection lattice of hyperplane arrangements. This condition is of the opposite form to that in Theorem \ref{thm: log threshold}. Furthermore, by taking $ S $ to be the zero set of
\begin{equation*}
	\prod\limits_{W \in \mathcal{L}}\prod\limits_{k=1}^{\dim W - 1}(\sum\limits_{L_{j} \in W}\alpha_{j}-\frac{k}{n}),
\end{equation*}
we obtain from Theorem \ref{thm: main} an explicit construction of the analytic subset in Theorem \ref{thm: generic multiplicity}, which was not provided in \cite{shi2024variation}.

The structure of the paper is as follows. In \S\ref{sec: preliminaries}, we recall Walther's criterion (Lemma \ref{lem: cohomological sufficient condition}) and reduce the problem to the computation of the cohomology of a certain local system on the hyperplane arrangement complement. In \S\ref{sec: connection} we construct extensions of certain connections and calculate their residues. In \S\ref{sec: proof 2} we prove the main theorem (Theorem \ref{thm: main}).

\textbf{Acknowledgments.} The second author is supported by the national key research and development program of China (No. 2022YFA1007100) and NSFC 12201337. We would like to thank Nero Budur for the wonderful lectures on singularities at YMSC, Tsinghua. We are also grateful for him pointing out \cite{budur2011local} to us. We also thank Quan Shi, Stephen Shing-Toung Yau, Huaiqing Zuo for organising the course and sharing notes.

\section{Preliminaries}\label{sec: preliminaries}

In this section, we review the known connections between Bernstein--Sato polynomials, Milnor fibers and hypersurface complements. We always assume that $ f \in \CC[x_{1},\cdots,x_{n}] $ is a homogeneous polynomial of degree $ d >0 $.

In \cite{walther2005bernstein, walther2017jacobian}, Walther established the relationships between the Bernstein--Sato polynomial $ b_{f}(s) $ and the Milnor fiber $ F = f^{-1}(1) $ of $ f $ at $ 0 $. In this paper, we will use the following sufficient cohomological condition to detect when $ -\frac{n}{d} $ is a root of $ b_{f}(s) $.

\begin{lem}[{\cite[Theorem 4.12]{walther2005bernstein}}]
	\label{lem: cohomological sufficient condition}
	Let $ f \in \CC[x_{1},\cdots,x_{n}] $ be a homogeneous polynomial of degree $ d $. Let $ F = f^{-1}(1) $ be the Milnor fiber of $ f $ at $ 0 $. Set
	\begin{equation*}
		\omega_{0} = \frac{1}{d}\sum\limits_{i=1}^{n} (-1)^{i-1} x_{i}dx_{1}\wedge \cdots \wedge \widehat{dx_{i}} \wedge \cdots \wedge dx_{n}.
	\end{equation*} 
	If the restriction of $ \omega_{0} $ on $ F $ defines a nonzero cohomological class in $ H^{n-1}(F,\CC) $, then $ -\frac{n}{d} $ is a root of the Bernstein--Sato polynomial $ b_{f}(s) $.
\end{lem}

We now calculate the cohomology of the Milnor fiber $ F $ of a homogeneous polynomial $ f $ by introducing a canonical $ \ZZ/d\ZZ $-action. For more details one can see \cite{cohen1995milnor} for reference.

Recall that $ F = f^{-1}(1) = \{(x_{1},\cdots,x_{n}) \in \CC^{n} \mid f(x_{1},\cdots,x_{n}) = 1 \} $ is the Milnor fiber of $ f $. Let  $ \zeta = e^{\frac{2\pi i}{d}} $ be the $ d $-th root of unity. Consider the following automorphism on $ F $:
\begin{equation*}
	\rho \colon F \longrightarrow F,\ (x_{1},\cdots,x_{n}) \mapsto (\zeta x_{1},\cdots,\zeta x_{n}).
\end{equation*}

Let $ D = \{f=0\} $ be the zero set of $ f $ in $ \PP^{n-1} $ and $ U = \PP^{n-1} \setminus D $ be its complement. The automorphism $ \rho $ induces a cyclic cover of degree $ d $:
\begin{equation*}
	p \colon F \longrightarrow U,\ (x_{1},\cdots,x_{n}) \mapsto [x_{1}:\cdots:x_{n}].
\end{equation*}

Since $ p \circ \rho = p $, the pull-back $ \rho^{*} \colon O_{F} \rightarrow \rho_{*}O_{F} $ induces an automorphism $ p_{*}O_{F} \rightarrow p_{*}O_{F} $, which we still denote by $ \rho^{*} $. Since $ \rho^{*} \colon p_{*}O_{F} \rightarrow p_{*}O_{F} $ commutes with the natural differential $ d_{F} $ on $ p_{*}O_{F} $, by the Riemann-Hilbert correspondence, it induces an automorphism $ \rho_{0} $ on the local system defined by the flat connection $ (p_{*}O_{F},d_{F}) $ on $ U $, which is exactly $ p_{*}\underline{\CC}_{F} $. Since $ p $ is a covering map and $ \rho_{0} $ is induced by $ \rho^{*} $, we have the following commutative diagram:
\begin{equation*}
	\begin{tikzcd}
		H^{n-1}(U,p_{*}\underline{\CC}_{F}) \arrow{d}{\rho_{0*}} \arrow{r}{\simeq} & H^{n-1}(F,\CC) \arrow{d}{\rho^{*}}  \\
		H^{n-1}(U,p_{*}\underline{\CC}_{F})	\arrow{r}{\simeq}	   & H^{n-1}(F,\CC) 
	\end{tikzcd}
\end{equation*}

Let $ \LL $ be the $ \zeta^{n} $-eigensheaf of $ \rho_{0} $ in $ p_{*}\underline{\CC}_{F} $. The above commutative diagram implies that $  H^{n-1}(U,\LL) \subseteq H^{n-1}(U,p_{*}\underline{\CC}_{F}) $ is exactly the $ \zeta^{n} $-eigenspace of $ \rho^{*} $ in $ H^{n-1}(F,\CC) $ under the identification $ H^{n-1}(U,p_{*}\underline{\CC}_{F}) \simeq H^{n-1}(F,\CC) $. In particular, since $ \rho^{*}\omega_{0} = \zeta^{n}\omega_{0} $, the cohomology class $ [\omega_{0}] $ is contained in $  H^{n-1}(U,\LL) $.

Let $ (\mathcal{V},\nabla) $ be the flat connection corresponding to $ \LL $ under the Riemann-Hilbert correspondence. Since $ \rho^{*} $ commutes with $ d_{F} $, we have that $ \mathcal{V} = O_{U}(n) $ is exactly the $ \zeta^{n} $-eigensheaf of $ \rho^{*} $ in $ p_{*}O_{F} $ and $ \nabla $ is the restriction of $ d_{F} $ on $ \mathcal{V} $. Furthermore, for any hyperplane $ H $ in $ \PP^{n-1} $ and its defining polynomial $ L $, since $ f_{L} = \frac{f}{L^{d}} $ defines a holomorphic function on $ U \setminus H $, we can identify $ \mathcal{V}|_{U \setminus H} $ with $ O_{U \setminus H}\cdot L^{n} $ and explicitly write the restriction of $ \nabla $ on $ U \setminus H $ as
\begin{equation}\label{eqn: connection on U}
	\nabla|_{U \setminus H} \colon  O_{U \setminus H}\cdot L^{n} \longrightarrow \Omega^{1}_{U \setminus H} \cdot L^{n},\ g \cdot L^{n} \mapsto(dg -  g \cdot \frac{n}{d}\cdot\frac{df_{L}}{f_{L}}) \cdot L^{n}.
\end{equation}

Since $ U $ is Stein, the cohomology of $ \LL $ is exactly the cohomology of the cochain complex
\begin{equation*}
	0 \longrightarrow	H^{0}(U,\mathcal{V}) \longrightarrow H^{0}(U,\Omega^{1}_{U} \otimes \mathcal{V}) \longrightarrow \cdots \longrightarrow H^{0}(U,\Omega^{n-1}_{U} \otimes \mathcal{V}) \longrightarrow 0,
\end{equation*}
where the differentials are given by $ \nabla $. Note that $ \omega_{0} $ can also be viewed as a global section of $ \Omega^{n-1}_{\PP^{n-1}} \otimes \V $ on $ \PP^{n-1} $. So the cohomology class $ [\omega_{0}] $ mentioned above is exactly the image of $ \omega_{0} $ under the following restriction map:
\begin{equation*}
	H^{0}(\PP^{n-1},\Omega^{n-1}_{\PP^{n-1}} \otimes \V ) \longrightarrow H^{0}(U,\Omega^{n-1}_{U} \otimes \mathcal{V}) \longrightarrow H^{n-1}(U,\LL).
\end{equation*}

\section{The extensions of connections}\label{sec: connection}

Assume that $ f = \prod\limits_{j=1}^{r} L_{j}^{a_{j}} (a_{j} \in \ZZ_{>0}) $ is a central hyperplane arrangement of degree $ d $ in $ \CC^{n} $. Keep the notation from Section \ref{sec: preliminaries}. In this section we calculate $ H^{p}(U,\LL) $ via the hypercohomology of the logarithmic de Rham complexes of certain extensions. Our purpose is to prove Proposition \ref{prop: algebraic de rham}.

Recall that $ D = \{f=0\} $ is the zero set of $ f $ in $ \PP^{n-1} $ and $ U = \PP^{n-1} \setminus D $ is its complement. We first review the construction in \cite{schechtman1995local} of a log resolution $ \pi \colon Y \rightarrow \PP^{n-1} $ of $ (\PP^{n-1},D) $.

Recall that $ A = \{L_{1},\cdots,L_{r}\} $ is the reduced hyperplane arrangement defined by $ f $. For any $ 1 \leq k \leq n-1 $, define $ \mathcal{L}_{k} $ to be the set of linear subspaces $ W \subsetneq (\CC^{n})^{*} $ satisfying that $ \dim W = k $ and $A \cap W $ is indecomposable in $ W $. Note that for any $W \in \mathcal{L}_k$, the set $ A \cap W $ spans $W$ since $A \cap W $ is indecomposable in $ W $. So $ \mathcal{L}_k $ is contained in the set consisting of the subspace spanned by $k$ functionals in $A$. This implies that $\mathcal{L}_k$ is a finite set.

For each $ 2 \leq k \leq n-1 $, $ \mathcal{L}_{k} $ defines a reduced subvariety in $ \PP^{n-1} $ of codimension $ k $, which we denote by $ Z_{k} $. Consider the sequence

	\begin{tikzcd}
		Y = Y_{1} \arrow{r}{\tau_{2}} & Y_{2}  \arrow{r}{\tau_{3}} & \cdots  \arrow{r}{\tau_{n-2}} & Y_{n-2}  \arrow{r}{\tau_{n-1}} & Y_{n-1} = \PP^{n-1},
	\end{tikzcd}

where $ \tau_{k} \colon Y_{k-1} \rightarrow Y_{k} $ is the blowing-up along the proper transform of $ Z_{k} $ under $ \tau_{n-1} \circ \cdots \circ \tau_{k+1} $.

\begin{lem}\label{lem: STV-blowup}\rm{(=\cite[Theorem 8]{schechtman1995local})}
	The composition $ \pi = \tau_{n-1} \circ \cdots \circ \tau_{2} \colon Y \rightarrow \PP^{n-1} $ is a blowing-up with centers in $ D $ such that $ Y $ is smooth and $ E = \pi^{-1}D $ is a normal crossing divisor on $ Y $.
\end{lem}

Recall that $ \mathcal{L} = \bigcup\limits_{k=1}^{n-1} \mathcal{L}_{k} $. By definition, each subspace $ W $ in $ \mathcal{L} $ defines an irreducible component $ E_{W} $ of $ E $ and $ E = \sum\limits_{W \in \mathcal{L}} E_{W} $.

In the sequel, we extend $ (\mathcal{V},\nabla) $ to connections on $ Y $ with a logarithmic pole along $ E $. 

Denote by $ O_{Y}[E] $ the sheaf of meromorphic functions on $ Y $ that are holomorphic on $ U $. For any coherent sheaf $ \mathcal{F} $ on $ Y $, denote by $ \mathcal{F}[E] = \mathcal{F}\otimes O_{Y}[E] $ for simplicity. Recall that the logarithmic de Rham complex of $ Y $ along $ E $ is defined to be the smallest subcomplex $ \Omega_{Y}(\log E) $ of $ \Omega_{Y}[E] $ containing $ \Omega_{Y} $ that is stable under the exterior product, and such that $ \frac{df}{f} $ is a local section of $ \Omega_{Y}(\log E) $ for all  local sections $ f $ of $  O_{Y}[E]^{\times} $. 

We first construct a specific extension of $ (\V,\nabla) $. Recall that for any hyperplane $ H $ in $ \PP^{n-1} $ and its defining polynomial $ L $, $ f_{L} = \frac{f}{L^{d}} $ defines a holomorphic function on $ \PP^{n-1} \setminus H $.  So the composition $ f_{L} \circ \pi $ is a holomorphic function on $ U_{H} = Y \setminus \pi^{-1}H $ and we can define a flat meromorphic connection on $ U_{H} $:
\begin{equation}\label{eqn: connection on Y}
	\nabla_{0,H} \colon O_{Y}[E]|_{U_{H}} \longrightarrow \Omega^{1}_{Y}[E]|_{U_{H}},\ g  \mapsto dg -  g \cdot \frac{n}{d}\cdot\frac{d(f_{L} \circ \pi)}{f_{L} \circ \pi}.
\end{equation}

Let $ \V_{0} = \pi^{*}O_{\PP^{n-1}}(n) $. Gluing all $ \nabla_{0,H} $ together yields a meromorphic connection:
\begin{equation*}
	\nabla_{0} \colon \V_{0}[E] \longrightarrow \Omega_{Y}^{1} \otimes \V_{0}[E].
\end{equation*}

The equation (\ref{eqn: connection on Y}) implies that $ \nabla_{0} $ sends $ \V_{0} $ to $ \Omega_{Y}^{1}(\log E) \otimes \V_{0} $. Furthermore, 	by comparing the equation (\ref{eqn: connection on Y}) with the equation ($ \ref{eqn: connection on U} $), we have the pair $ (\V_{0},\nabla_{0}) $ is exactly an extension of $ (\V,\nabla) $ on $ Y $ with a logarithmic pole along $ E $.

By twisting $ \V_{0} $, we obtain more extensions of $ (\V,\nabla) $. For any map $ \mu \colon  \mathcal{L} \rightarrow \ZZ $, consider the following locally free subsheaf of $ \V_{0}[E] $:
\begin{equation*}
	\mathcal{V}_{\mu} = \mathcal{V}_{0} \otimes O_{Y}(-\sum\limits_{W \in \mathcal{L}}\mu(W)E_{W}).
\end{equation*}

\begin{lem}\label{lem: extension of connection}
	The pair $ (\V_{\mu},\nabla_{0}) $ is an extension of $ (\V,\nabla) $ on $ Y $ with a logarithmic pole along $ E $, whose residue along $ E_{W} $ is exactly
	\begin{equation*}
		\mu(W) - \frac{n}{d}\sum\limits_{L_{j} \in W}a_{j}.
	\end{equation*}
\end{lem}

\begin{proof}
	For any point $ y $ in $ Y $, let $ U_{y} $ be a small neighbourhood of $ y $ and choose a local coordinate $ \{z_{1},\cdots,z_{n-1}\} $ such that $ \pi^{-1}D $ is locally given by $ z_{1}\cdots z_{m} = 0 $ in $ U_{y} $. Let $ E_{W_{i}} $ be the irreducible component of $ \pi^{-1}D $ such that $ E_{W_{i}} \cap U_{y} = \{z_{i}=0\} $. Set $ b_{i} = \mu(W_{i}) $. Fix a generic hyperplane $ H $ such that $ U_{y} \cap \pi^{-1}H = \emptyset $ and let $ L $ be the defining polynomial of $ H $.  
	
	Under the identification $ \mathcal{V}_{0}|_{U_{y}} = O_{U_{y}}$, we have
	\begin{equation*}
		\mathcal{V}_{\mu}|_{U_{y}} = O_{U_{y}} \cdot z_{1}^{b_{1}}\cdots z_{m}^{b_{m}}.
	\end{equation*}
	
	And the restriction of the holomorphic function $ \frac{L_{j}}{L} \circ \pi $ on $ U_{y} $ has the form
	\begin{equation*}
		(\frac{L_{j}}{L} \circ \pi)|_{U_{y}} = g_{j}\prod\limits_{L_{j} \in W_{i}}z_{i},
	\end{equation*}
	where $ g_{j} $ is a holomorphic function on $ U_{y} $ such that $ g_{j}(y) \neq 0 $.

	Then for any holomorphic function $ g $ on some open subset of $ U_{y} \subseteq U_{H} $, by the equation (\ref{eqn: connection on Y}) we have
	\begin{eqnarray*}
		\nabla_{0}(g\cdot z_{1}^{b_{1}}\cdots z_{m}^{b_{m}}) & = & d(g\cdot z_{1}^{b_{1}}\cdots z_{m}^{b_{m}}) -  g \cdot \frac{n}{d} \cdot \frac{d(f_{L} \circ \pi)}{f_{L} \circ \pi} \cdot z_{1}^{b_{1}}\cdots z_{m}^{b_{m}} \\
		& = & \left(dg + g\cdot \sum\limits_{i=1}^{m}b_{i}\frac{dz_{i}}{z_{i}}-g\cdot \frac{n}{d}\cdot\frac{d(f_{L} \circ \pi)}{f_{L} \circ \pi}\right)\cdot z_{1}^{b_{1}}\cdots z_{m}^{b_{m}} \\
		& = & \left(dg + g\cdot \left(\sum\limits_{i=1}^{m}\left(b_{i}-\frac{n}{d}\sum\limits_{L_{j} \in W_{i}}a_{j}\right)\frac{dz_{i}}{z_{i}}- \frac{n}{d}\cdot\sum\limits_{j=1}^{r}a_{j}\frac{dg_{j}}{g_{j}}\right)\right)\cdot z_{1}^{b_{1}}\cdots z_{m}^{b_{m}}.
	\end{eqnarray*}

	Therefore, the image of $ \mathcal{V}_{\mu} $ under $ \nabla_{0} $ is contained in $ \Omega_{Y}^{1}(\log E) \otimes \mathcal{V}_{\mu} $, which implies that the pair $ (\V_{\mu},\nabla_{0}) $ is an extension of $ (\V,\nabla) $ on $ Y $ with a logarithmic pole along $ E $. Furthermore, for any $ W \in \mathcal{L} $, take $ y $ as a generic point on $ E_{W} $. In this case, we have $ m = 1 $ and $ W_{1} = W $. So the residue of $ (\mathcal{V_{\mu}},\nabla_{0}) $ along $ E_{W} $ is exactly 
	\begin{equation*}
		\mu(W)-\frac{n}{d}\sum\limits_{L_{j} \in W}a_{j}.
	\end{equation*}
\end{proof}

Now the following proposition is a direct corollary of  Lemma \ref{lem: extension of connection} and the algebraic de Rham theorem (see \cite[Corollary 6.10]{deligne1970equations}).

\begin{prop}\label{prop: algebraic de rham}
	Let $ \mu \colon \mathcal{L} \rightarrow \ZZ $ be a map satisfying that for any $ W \in \mathcal{L} $,
	\begin{equation*}
		\mu(W)-\frac{n}{d}\sum\limits_{L_{j} \in W}a_{j} \notin \ZZ_{>0}.
	\end{equation*}
	Then for any $ p \geq 0 $,
	\begin{equation*}
		\HH^{p}(Y,\Omega_{Y}^{\cdot}(\log \pi^{-1}D)\otimes \mathcal{V_{\mu}}) \simeq H^{p}(U,\LL).
	\end{equation*}
\end{prop}

\section{Proof of Theorem \ref{thm: main}}\label{sec: proof 2}

In this section, we will prove Theorem \ref{thm: main}. We will need the following two technical lemmas. The first one is a translation of indecomposability. Although this is implicit in \cite[Lemma 4.4-4.6]{budur2024motivic}, it is not stated explicitly there. We formulate it as follows and give a concise proof.
\begin{lem}\label{lem: combinatoric translation for indecomposable}
	Let $ V $ be an $ n $-dimensional $ \CC $-linear space($n \geq 2$) and $ A = \{ L_{1},\cdots,L_{r} \} $ be an indecomposable subset of $ V $ consisting of nonzero vectors. Then there exist positive rational numbers $ \epsilon_{1},\cdots,\epsilon_{r} $ such that $ \sum\limits_{j=1}^{r} \epsilon_{j} = n $ and for any nonzero proper subspace $ W \subsetneq V $,
	\begin{equation*}
		\sum\limits_{L_{j} \in W} \epsilon_{j} < \dim W.
	\end{equation*}
\end{lem}
\begin{proof}
	Since $ A $ is indecomposable, $ L_{1},\cdots,L_{r} $ span the whole space $ V $. So without loss of generality we may assume $ B = \{L_{1},\cdots,L_{n}\} $ is a basis of $ V $. Then for any $ j > n $ we can define $ B_{j} $ to be the subset of $ B $ consisting of those $ L_{i} $ such that the coefficient of $ L_{i} $ is not zero in the linear representation of $ L_{j} $. For each $ 1 \leq i \leq n $, set $ b_{i} = \# \{j>n \mid L_{i} \in B_{j}\}$.
	
	Take
	\begin{equation*}
		\epsilon_{j} = \left\{\begin{array}{ll}
			\frac{n}{n+1}, & 1 \leq j \leq n \\[2mm]
			(\sum\limits_{L_{i} \in B_{j}}\frac{1}{b_{i}})\cdot \frac{1}{n+1}, & n+1 \leq j \leq r.		
		\end{array}
		\right.
	\end{equation*}
	Obviously $ \epsilon_{1},\cdots,\epsilon_{r} $ are positive rational numbers satisfying that
	\begin{equation*}
		\sum\limits_{j=1}^{r}\epsilon_{j} = \frac{n^{2}}{n+1} + \sum\limits_{j=n+1}^{r}(\sum\limits_{L_{i} \in B_{j}}\frac{1}{b_{i}})\cdot \frac{1}{n+1} = \frac{n^{2}}{n+1} + \sum\limits_{i=1}^{n}(\sum\limits_{\substack{j > n \\ L_{i} \in B_{j}}}\frac{1}{b_{i}})\cdot \frac{1}{n+1} = \frac{n^{2}+n}{n+1} = n.
	\end{equation*}
	
	Furthermore, for any nonzero proper subspace $ W \subsetneq V $, consider the intersection $ W \cap B $. If $ \#(W \cap B) \leq \dim W - 1 $, then we have
	\begin{eqnarray*}
		\sum\limits_{L_{j} \in W} \epsilon_{j} & \leq & (\dim W - 1) \cdot \frac{n}{n+1} + \frac{n}{n+1} \\
		& = & \dim W \cdot \frac{n}{n+1} \\
		& < & \dim W.
	\end{eqnarray*}
	
	Otherwise, we have $ W \cap B $ is a basis of $ W $ since $ B $ is linearly independent. So for any $ j > n $, $ L_{j} \in W $ if and only if $ B_{j} \subset W $. Furthermore, since $ A $ is indecomposable, there exists some $  j_{0} > n $ such that
	\begin{equation*}
		B_{j_{0}} \cap W \neq \emptyset,\ \text{and\ }B_{j_{0}} \nsubseteq W.
	\end{equation*} 
	
	So we have
	\begin{equation*}
		\sum\limits_{\substack{j > n \\ B_{j} \subset W}}(\sum\limits_{L_{i} \in B_{j}}\frac{1}{b_{i}}) < \sum\limits_{\substack{j > n \\ B_{j} \cap W \neq \emptyset}}(\sum\limits_{L_{i} \in B_{j} \cap W}\frac{1}{b_{i}}) =  \sum\limits_{\substack{1 \leq i \leq n \\ L_{i} \in W}}(\sum\limits_{\substack{j > n \\ L_{i} \in B_{j}}} \frac{1}{b_{i}})  = \#(W \cap B) = \dim W.
	\end{equation*}
	
	Therefore
	\begin{eqnarray*}
		\sum\limits_{L_{j} \in W} \epsilon_{j} & = & \dim W \cdot \frac{n}{n+1} + 	\sum\limits_{\substack{j > n \\ B_{j} \subset W}}(\sum\limits_{L_{i} \in B_{j}}\frac{1}{b_{i}})\cdot\frac{1}{n+1} \\
		& < & \dim W \cdot \frac{n}{n+1} + \frac{\dim W}{n+1} \\
		& = & \dim W.
	\end{eqnarray*}
	
	So $ \epsilon_{1},\cdots,\epsilon_{r} $ meet the requirements.
\end{proof}

The second lemma is a useful vanishing theorem from \cite{esnault1992lectures}.
\begin{lem}\label{lem: vanishing}\rm{(=\cite[Theorem 6.2]{esnault1992lectures})}
	Let $ Y $ be a $ n $-dimensional projective manifold over $ \CC $. Let $ E = \sum\limits_{j=1}^{r} E_{j} $ be a reduced normal crossing divisor such that $ Y \setminus E $ is affine. Let $ \mathcal{V} $ be an invertible sheaf on $ Y $ such that there exist positive integers $ c_{1},\cdots,c_{r},N $ satisfying that $ 0 < c_{j} < N $ and $ \mathcal{V}^{N} = O_{Y}(\sum\limits_{j=1}^{r}c_{j}E_{j}) $. Then for any $ p+q \neq n $,
	\begin{equation*}
		H^{p}(Y,\Omega_{Y}^{q}(\log E)\otimes \mathcal{V}^{-1}) = 0.
	\end{equation*}
\end{lem}

Now we are ready to prove the main theorem. We keep the notation from Section \ref{sec: connection}.

\begin{proof}[Proof of Theorem \ref{thm: main}]
	
	By Lemma \ref{lem: combinatoric translation for indecomposable}, there exist positive rational numbers $ \epsilon_{1},\cdots,\epsilon_{r} $ such that $ \sum\limits_{j=1}^{r} \epsilon_{j} = n $ and for any nonzero proper subspace $ W \subsetneq (\CC^{n})^{*} $,
	\begin{equation*}
		\sum\limits_{L_{j} \in W} \epsilon_{j} < \dim W.
	\end{equation*}
	
	By perturbing $ \epsilon_{j} $, we may assume for any $ W \in \mathcal{L} $, the number $ \sum\limits_{L_{j} \in W} \epsilon_{j} $ is not an integer. Denote by
	\begin{equation*}
		\mu(W) = 1+ \lfloor  \sum\limits_{L_{j} \in W} \epsilon_{j} \rfloor.
	\end{equation*}
	By definition, $ 1 \leq \mu(W) \leq \dim W $ and $ 0 < 1+  \sum\limits_{L_{j} \in W} \epsilon_{j} - \mu(W) < 1 $.

	Let $ \pi \colon Y \rightarrow \PP^{n-1} $ be the blowing-up defined in Lemma \ref{lem: STV-blowup} and denote by $ E = \pi^{-1}D $. By Lemma \ref{lem: extension of connection}, there is an extension $ (\mathcal{V}_{\mu},\nabla_{0}) $ of $ (\V,\nabla) $ with a logarithmic pole along $ E $ such that 
	\begin{equation*}
		\mathcal{V}_{\mu} = O_{Y}(-\sum\limits_{W \in \mathcal{L}}\mu(W)E_{W})\otimes\pi^{*}O_{\PP^{n-1}}(n).
	\end{equation*}

	Furthermore, for any $ W \in \mathcal{L} $, the residue of $ (\mathcal{V}_{\mu},\nabla_{0}) $ along $ E_{W} $ is 
	\begin{equation*}
		\mu(W) - \frac{n}{d}\sum\limits_{L_{j} \in W}a_{j} \leq \dim W - \frac{n}{d}\sum\limits_{L_{j} \in W}a_{j}.
	\end{equation*}
	
	Since $ f $ satisfies the condition (\ref{eqn: nonresonant}), the right hand side is not a positive integer. So the residue of $ (\mathcal{V}_{\mu},\nabla_{0}) $ along $ E_{W} $ is not a positive integer. So by Proposition \ref{prop: algebraic de rham}, we have
	\begin{equation*}
		\HH^{n-1}(Y,\Omega^{\cdot}_{Y}(\log E)\otimes\mathcal{V}_{\mu}) \simeq H^{n-1}(U,\LL).
	\end{equation*}

	Consider the hypercohomology spectral sequence
	\begin{equation*}
		E^{pq}_{1} = H^{p}(Y,\Omega^{q}_{Y}(\log E)\otimes\mathcal{V}_{\mu}) \Rightarrow \HH^{p+q}(Y,\Omega^{\cdot}_{Y}(\log E)\otimes\mathcal{V}_{\mu}).
	\end{equation*}
	
	Choose $  N \in \NN $ such that $  N\epsilon_{j}(1 \leq j \leq r) $ are all integers. Let $ D_{j} $ be the reduced divisor defined by $ L_{j} $ in $ \PP^{n-1} $. Note that
    \begin{equation*}
        O_{\PP^{n-1}}(-nN) = O_{\PP^{n-1}}\left(-\sum\limits_{j=1}^{r}N\epsilon_{j}D_{j}\right).
    \end{equation*}

    Then we have
\begin{eqnarray*}
\mathcal{V}_\mu^{-N} & = & \pi^{*}O_{\PP^{n-1}}(-nN)\otimes O_{Y}\left(\sum\limits_{W \in \mathcal{L}}N\mu(W)E_{W}\right) \\
        &  = & O_{Y}\left(-\sum\limits_{j=1}^{r}N\epsilon_{j}\pi^{*}D_{j} +\sum\limits_{W \in \mathcal{L}}N\mu(W)E_{W}\right) \\
        & = & O_{Y}\left(-\sum\limits_{j=1}^{r}N\epsilon_{j}\left(\sum\limits_{W \ni L_j}E_{W}\right) +\sum\limits_{W \in \mathcal{L}}N\mu(W)E_{W}\right) \\
		& = & O_{Y}\left(\sum\limits_{W \in \mathcal{L}}N\left(\mu(W)-\sum\limits_{L_{j} \in W}\epsilon_{j}\right)E_{W}\right).
	\end{eqnarray*}
	
	Since $ N(\mu(W)-\sum\limits_{L_{j} \in W}\epsilon_{j}) $ is an integer satisfying that
	\begin{equation*}
		0 <  N(\mu(W)-\sum\limits_{L_{j} \in W}\epsilon_{j}) < N,
	\end{equation*}
	by Lemma \ref{lem: vanishing} we have for any $ p+q \neq n-1 $,
	\begin{equation*}
		H^{p}(Y,\Omega_{Y}^{q}(\log E)\otimes \mathcal{V}_\mu) = 0.
	\end{equation*}

	So the edge morphism
	\begin{equation*}
		H^{0}(Y,\Omega^{n-1}_{Y}(\log E)\otimes\mathcal{V}_{\mu}) \longrightarrow \HH^{n-1}(Y,\Omega^{\cdot}_{Y}(\log E)\otimes\mathcal{V}_{\mu})
	\end{equation*}
	is injective.
	
	Since
	\begin{equation*}
		K_{Y} = \pi^{*}K_{\PP^{n-1}} + \sum\limits_{W \in \mathcal{L}}(\dim W-1)E_{W} = \pi^{*}K_{\PP^{n-1}} - E + \sum\limits_{W \in \mathcal{L}} \dim W \cdot E_{W},
	\end{equation*}
	we have 
	\begin{eqnarray*}
		H^{0}(\PP^{n-1},\Omega_{\PP^{n-1}}^{n-1}\otimes O_{\PP^{n-1}}(n)) & = & H^{0}(Y,\pi^{*}\Omega_{\PP^{n-1}}^{n-1}\otimes \pi^{*}O_{\PP^{n-1}}(n)) \\
		& = & H^{0}\left(Y,\Omega^{n-1}_{Y}\otimes O_{Y}(E-\sum\limits_{W \in \mathcal{L}}\dim W \cdot E_{W})\otimes \pi^{*}O_{\PP^{n-1}}(n)\right) \\
		& \subseteq & H^{0}\left(Y,\Omega^{n-1}_{Y}\otimes O_{Y}(E-\sum\limits_{W \in \mathcal{L}} \mu(W) E_{W})\otimes \pi^{*}O_{\PP^{n-1}}(n)\right) \\
		& = & H^{0}(Y,\Omega^{n-1}_{Y}(\log E)\otimes \mathcal{V}_{\mu}). 
	\end{eqnarray*}
	
	So
	\begin{equation*}
		\omega_{0} \in H^{0}(\PP^{n-1},\Omega_{\PP^{n-1}}^{n-1}\otimes O_{\PP^{n-1}}(n)) \subseteq H^{0}(Y,\Omega^{n-1}_{Y}(\log E)\otimes \mathcal{V}_{\mu})
	\end{equation*} 
	is a nonzero section, which implies that $ [\omega_{0}] \neq 0 $ in $  \HH^{n-1}(Y,\Omega^{\cdot}_{Y}(\log\pi^{-1}D)\otimes\mathcal{V}_{\mu}) \simeq H^{n-1}(U,\LL) $. As we discussed in Section \ref{sec: preliminaries}, this implies that $ \omega_{0} $ defines a nonzero cohomology class in $ H^{n-1}(F,\CC) $, where $ F=f^{-1}(1) $ is the Milnor fiber of $ f $ at $ 0 $. So by Lemma \ref{lem: cohomological sufficient condition}, the $ \frac{n}{d} $-conjecture holds for $ f $.
\end{proof}


\bibliographystyle{plain}
\bibliography{pamq-bibliography}          
%
%

\end{document}